\documentclass[12pt]{amsart}

\usepackage{latexsym,mathrsfs,bm}
\usepackage[english]{babel}
\usepackage{paralist}
\usepackage{ulem} 
\usepackage{amsfonts,amssymb,amsmath}
\usepackage{graphicx}
\usepackage{subcaption}

\usepackage{amsthm}
\usepackage[latin1]{inputenc}
\usepackage{bm}

\newcommand{\eps}{\varepsilon}
\newcommand{\R}{\mathbb{R}}
\newcommand{\Q}{\mathbb{Q}}
\newcommand{\C}{\mathbb{C}}

\newcommand{\Z}{\mathbb{Z}}
\newcommand{\A}{\mathfrak{A}}
\newcommand{\es}[1]{\begin{equation}\begin{split}#1\end{split}\end{equation}}
\newcommand{\est}[1]{\begin{equation*}\begin{split}#1\end{split}\end{equation*}}
\newcommand{\as}[1]{\begin{align}#1\end{align}}

\newcommand{\tn}[1]{\textnormal{#1}}
\newcommand{\sumprime}{\sideset{}{'}\sum}

\renewcommand{\mod}[1]{~\pr{\textnormal{mod}~#1}}
\newtheorem*{theo*}{Theorem}
\newtheorem{theo}{Theorem}
\newtheorem{theorem}[theo]{Theorem}
\newtheorem{ezer}{Exercise}

\newtheorem{prop}[ezer]{Proposition}

\newtheorem{lemma}{Lemma}
\newtheorem{corol}[lemma]{Corollary}
\newtheorem{remark}{Remark}

\newtheorem*{rem*}{Remark}

\newcommand{\pr}[1]{\left( #1\right)}

\newcommand{\pa}[1]{\left\langle #1\right\rangle}

\newcommand{\pg}[1]{\left\{ #1\right\}}

\newcommand{\e}[1]{\operatorname{e}\pr{ #1}}
\newcommand{\cc}{\operatorname{c}}

\newcommand{\comment}[1]{}

\let\originalleft\left
\let\originalright\right
\renewcommand{\left}{\mathopen{}\mathclose\bgroup\originalleft}
\renewcommand{\right}{\aftergroup\egroup\originalright}

\numberwithin{equation}{section}
\begin{document}

\author{Sandro Bettin}
\title{On the distribution of a cotangent sum}
\date{}
\address{Sandro Bettin --  Centre de Recherches Math\'{e}matiques - Universit\'{e} de Montr\'{e}al, P.O. Box 6128, Centre-ville Station, Montr\'{e}al, QC, H3C 3J7, Canada}

\maketitle

\begin{abstract}
Maier and Rassias computed the moments and proved a distribution result for the cotangent sum $\cc_0\pr{ a/q}:=-\sum_{m<q}\frac mq\cot\pr{\frac{\pi  ma}{q}}$ on average over  $1/2<A_0\leq a/q<A_1<1$, as $q\rightarrow \infty$. We give a simple argument that recovers their results (with stronger error terms)
 and extends them to the full range $1\leq a<q$. Moreover, we give a density result for $\cc_0$ and answer a question posed by Maier and Rassias on the growth of the moments of $\cc_0$.
\end{abstract}

\section{Introduction}

In this note, we consider the cotangent sum
\est{
\cc_0\pr{ a/q}:=-\sum_{m=1}^{q-1}\frac mq\cot\pr{\frac{\pi  ma}{q}},\qquad (a,q)=1,q\geq1,
}
which is related to the Nyman-Beurling criterion for the Riemann hypothesis (see, for example,~\cite{Bag,BC}) and was recently studied in~\cite{BC} and~\cite{MR}. In~\cite{BC}, Conrey and the author proved that $\cc_0$ satisfies the reciprocity formula
\es{\label{rf}
\cc_0(a/q)+q/a\cc_0(q/a)-(\pi q)^{-1} = \psi(a/q),
}
where $\psi(x)$ is an analytic function in $\C\setminus\R_{ \leq 0}$ and satisfies several nice properties. In particular, it satisfies the three term relation $\psi(x)=\psi(x+1)+(x+1)^{-1}\psi(x/(x+1))$ and has an asymptotic expansion for $x\rightarrow0$, starting by
\es{\label{aps}
\psi(x)=-\frac{\log (2\pi x)-\gamma}{\pi x}+O(\log x).
}

Ishibashi~\cite{Ish} observed that $\cc_0$ is also related to the value at $s=0$, or $s=1$ by the functional equation, of the (``imaginary part'' of the) Estermann function:
\es{\label{sf}
\cc_0\pr{ a/q}=\frac12D_{\sin}(0, a/q)=2q\pi^{-2}D_{\sin}(1,\overline a/q),
}
where for $x\in\R$, $\Re(s)>1$,
\est{
D_{\sin}(s,x):=\sum_{n=1}^\infty\frac{d(n)\sin(2\pi nx)}{n^s}
}
with $d(n)$ indicating the divisor function. If $x\in\Q$, then $D_{\sin}(s,x)$ can be extended to an entire function of $s$ satisfying the functional equation
\es{\label{fe}
\Lambda_{\sin}\pr{s,a/q}:=\Gamma\pr{\tfrac {1+s}2}^2 \pr{{q}/{\pi}}^{s} D_{\sin}\pr{s, a/q}
=\Lambda_{\sin}\pr{1-s, {\overline a }/q},\\
}
where $\overline a$ denotes the inverse of $a$ modulo the denominator $q$.

If $x\in\R\setminus\Q$, then de la Bret\`eche and Tenenbaum~\cite{dBT} showed that the convergence of the series at $s=1$ is equivalent to the convergence of $ \sum_{n\geq1} (-1)^n{\log v_{n+1}}/{v_n},$
where $u_n/v_n$ denotes the $n$-th partial quotient of $x$. Moreover, they also showed that for $x\not\in\Q$
\est{
D_{\sin}(1,x):=\sum_{n=1}^\infty \frac{d(n)\sin(2\pi nx)}{n}=\pi\sum_{n=1}^\infty \frac{\tfrac12-\{nx\}}{n},
}
whenever one of the two series converges, and where $\{x\}$ denotes the fractional part of $x$.

Recently, in a difficult paper Maier and Rassias~\cite{MR} computed the moments of $\cc_0\pr{ a/q}$ proving that
\es{\label{mrmo}
\frac1{\varphi(q)}\sum_{\substack{(a,q)=1,\\A_0<{a}/{q}<A_1}}\cc_0\pr{ a/q}^{k}=H_kq^k(1+o(1))
}
as $q\rightarrow\infty$, for certain constants $H_k$ and any fixed $\frac12<A_0<A_1<1$ and where $\phi(q)$ is Euler's function,. They also computed the distribution of $\tfrac1q\cc_0\pr{ a/q}$ and proved that 
\es{\label{fvfsdx}
\frac1{(A_1-A_0)\varphi(q)}\sum_{\substack{(a,q)=1,\\A_0<{a}/{q}<A_1}}f\pr{\tfrac1q\cc_0\pr{ a/q}}=(1+o(1))\int_{\R}f(x)dF(x),
}
as $q\rightarrow\infty$ for any continuous function of compact support $f(x)$ and where $F(x)$ is the continuous (as it is shown in the same paper) function defined by
\est{
F(x):=\tn{meas}\pr{\{z\in [0,1]\mid 2\pi^{-2}D(1,z)\leq x\}}.
}
\begin{figure}[!!h]
\begin{subfigure}[b]{0.5\textwidth}
\centering
\includegraphics[width=190pt]{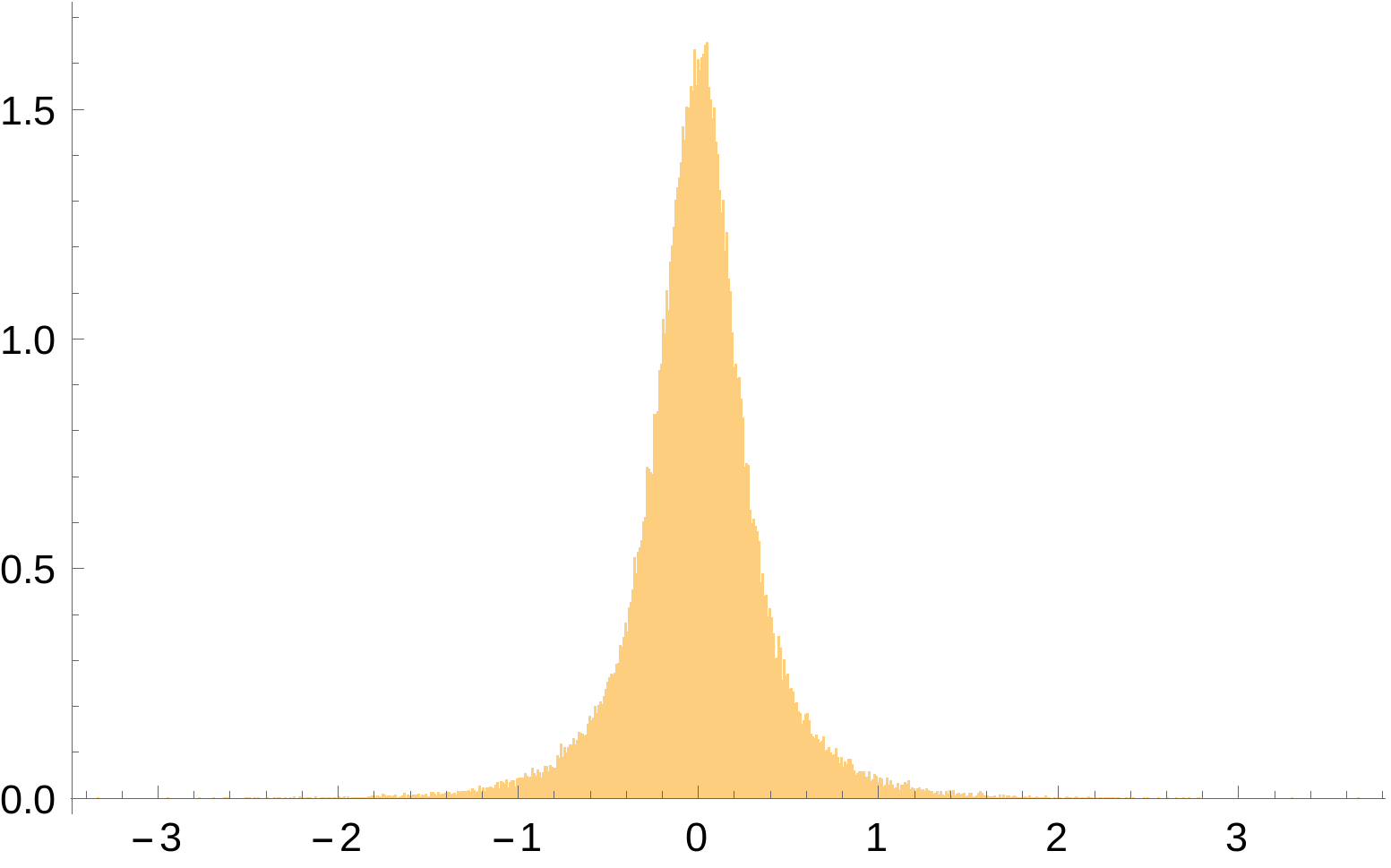}
\end{subfigure}%
\begin{subfigure}[b]{0.5\textwidth}
\centering
\includegraphics[width=190pt]{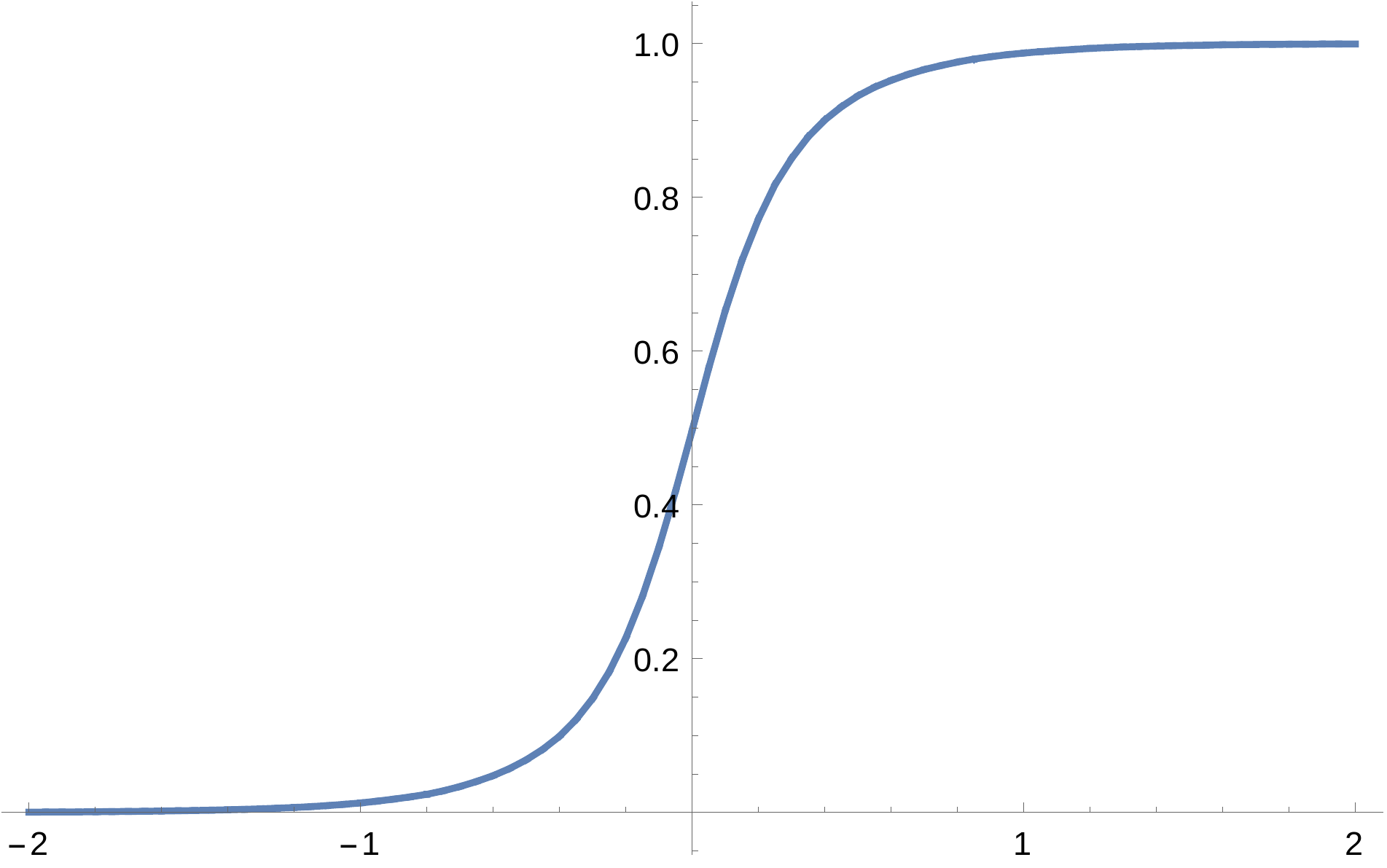}
\end{subfigure}%
\caption{The histogram of $2\pi^{-2}D(1,x)$ and the graph of $F(x)$ obtained by sampling $D(1,x)$ (truncated at $n\leq 10^5$) at $10^5$ points chosen uniformly in $[0,1]$.}
\end{figure}

In this note we extend the results of~\cite{MR} to the full range $A_0=0,A_1=1$. Moreover, since we are averaging over the full range, our results immediately give moments and distribution also for $D(1,a/q)$ or, equivalently, for the Vasyunin sum $V(a,q):=-\cc_0(\overline a/q)$.

Our method, with the orthogonality for additive characters replaced by Weil's bound (as in Lemma~8 of~\cite{DFI}), gives also the case when $A_0\neq0,A_1\neq1$, thus providing a new simpler proof of the results of Maier and Rassias with stronger bounds for the error terms. In particular, we obtain~\eqref{mrmo} with the error term $o(1)$ replaced by $O_\eps(q^{k-\frac12+\eps}(Ak\log q)^{2k})$, allowing us to handle the case of short intervals $[A_0,A_1]$ with $A_1-A_0\gg q^{-\frac12+\delta}$ for some $\delta>0$. However, since the details are identical, we limit ourselves to the full range $A_0=0,A_1=1$.

\begin{theorem}\label{cstare}
Let $q\geq1$ and $k\geq0$. Then, 
\es{\label{fvfsdx2}
\frac1{\varphi(q)}\sum_{\substack{a=1,\\(a,q)=1}}^{q}\cc_0\pr{ a/q}^{k}=H_kq^k+O_\eps(q^{k-1+\eps}(Ak\log q)^{2k}),
}
for some absolute constant $A>0$ and any $\eps>0$, where 
\est{
H_k:=(i\pi^{2})^{-k}\hspace{-1em}\sum_{\substack{(n_1,\dots,n_k)\in(\Z_{\neq0})^{k},\\n_1+\cdots+n_k= 0}} \frac{d(|n_1|)\cdots d(|n_k|)}{n_1\cdots n_k}.
}
\end{theorem}
\begin{remark}
If $k$ is odd, then both $H_k$ and the left hand side of~\eqref{fvfsdx2} are identically zero. 
\end{remark}

In the same paper Maier and Rassias ask whether $\sum_{k=0}^\infty \frac{H_kt^k}{k!}$ has a positive radius of convergence. The following Theorem answers their question in the affirmative.

\begin{theorem}\label{cbg}
We have $H_k\ll A^k k!$ for some $A>0$.
\end{theorem}
Since  $\sum_{k=0}^\infty H_k t^k/k!$ has a positive radius of convergence, we immediately obtain the distribution of $\cc_0$ over the full range.
\begin{corol}\label{mcrm}
Let $q\geq1$ and let $f(x)$ be a continuous function with compact support. Then, as $q\rightarrow\infty$ we have
\est{
\frac1{\varphi(q)}\sum_{\substack{a=1,\\(a,q)=1}}^{q}f\pr{\tfrac{1}{q}\cc_0\pr{\tfrac aq}}=(1+o(1))\int_{\R}f(x)dF(x).
}
\end{corol}

Using~\eqref{rf}, we can also give an alternative expression for $D_{\sin}(1,x)$ in terms of the denominators of the partial quotient of $x$.  
\begin{prop}\label{mp}
Let $\pa{a_0;a_1,a_2,\dots}$ be the continued fraction expansion of $x\in\R$. Moreover, let $u_r/v_r$ be the $r$-th partial quotient of $x$. Then
\es{\label{csa}
D_{\sin}(1,x):=\sum_{n}\frac{d(n)\sin(2\pi nx)}{n}
=-\frac{\pi^2}2\sum_{\ell=1}^{\infty}\frac{(-1)^\ell}{v_\ell}\pr{\frac1{\pi v_{\ell}}+\psi\pr{\frac{v_{\ell-1}}{v_\ell}}},
}
whenever either of the two series is convergent. 

If $x=\pa{a_0;a_1,\dots,a_r}$ is a rational number, then  the range of summation of the series on the right is to be interpreted to be $1\leq \ell\leq r$.
\end{prop}
\begin{remark}
If $x\in\Q$ then one can write two different continued fraction expansion for $x$, but~\eqref{csa} holds regardless of the chosen expansion.
\end{remark}

Proposition~\ref{mp}, which constitutes a refinement of the aforementioned work of de la Bret\`eche and Tenenbaum, can be interpreted as an extension of the reciprocity formula~\eqref{rf} to $x\notin\Q$. We also remark that Proposition~\ref{mp} is of independent interest as $D(1,a/q)=-\frac{\pi^2}{2q}V(a,q)$ is exactly the sum appearing in the Nyman-Beurling criterion for the Riemann hypothesis (c.f.~\cite{BC}).

The exact formula~\eqref{csa} allows us to prove the following corollary which, combined with~\eqref{fvfsdx} and the periodicity modulo 1 of $\cc_{0}$, implies that $\{\big(a/q,\frac1qc_0(a/q)\big) \mid (a,q)=1,q\geq1 \}$ is dense in $\R^2$.

\begin{corol}\label{finc}
The function $F(x)$ is strictly increasing.
\end{corol}
\begin{corol}\label{den}
We have that $\{\big(a/q,\frac1qc_0(a/q)\big) \mid (a,q)=1,q\geq1 \}$ is dense in $\R^2$.
\end{corol}

\section*{Acknowledgments}
After sending them a preprint of this paper, the author was informed by Maier and Rassias that they have also obtained independently Theorem 2.1 using a somewhat similar argument.

The author would like to thank Maksym Radziwi\l\l~for useful discussions.

\section{Proof of Theorem~\ref{cstare}}

Both in this and in the following section, we will consider $D(1,a/q)$ rather than $\cc_0(a/q)$. The stated result then follows by~\eqref{sf}. Moreover, we assume that $k$ is even, as the result is trivial otherwise.

First, we observe that we can have 
\es{\label{tgt}
D_{\sin}(1,a/q)
&=\sum_{n\leq 2X}g_X( n)\frac{d(n)\sin(2\pi na/q)}{n} +O(q^{1+2\eps_1}/X),\\
}
for $\eps_1=0.1$ and where $g_X(x)$ is a bounded function supported in $[0,2X]$ and  identically $1$ for $1\leq x\leq X$. This can be seen by taking a smooth partition of unity satisfying
\est{
1=\sumprime_{M}F(x/M)\quad \forall x\in[1/2,\infty),\qquad \sumprime_{Y_1\leq M\leq Y_2}1\ll \log (2+Y_2/Y_1)\quad 1\leq Y_1\leq Y_2,\\[-0.3em]
}
where $F(x)$ is smooth, supported in $[1/2,1]$, and satisfying $F^{(j)}(x)\ll_j 1$ for all $j\geq 0$ (so that the Mellin transform $\hat F(s)$ of $F(x)$ is entire and decays rapidly on vertical strips). Then, writing $F$ in terms of its Mellin transform $\hat F(s)$, we have
\est{
\sum_{n\geq1}F\pr{\frac nM}\frac{d(n)\sin(2\pi na/q)}{n}&=\int_{(1)}\hat F(s)D(1+s,\tfrac aq)M^{s}ds=\int_{(-1- {\eps_1})}\hat F(s)D(1+s,\tfrac aq)M^{s}ds\\
&\ll q^{1+2\eps_1} M^{-1-{\eps_1}}
}
as can be seen by using~\eqref{fe} and bounding trivially. Thus,~\eqref{tgt} follows by taking $g_X(x):=\sum_{M\leq 2X}F(x/M)$.

By Euler's formula, when $X\geq q^{1+2\eps_1}$~\eqref{tgt} gives
\est{
D_{\sin}(1, a/q)^k=(2i)^{-k}\hspace{-1em}\sum_{\substack{n_1,\dots,n_k\in B^*_{2X}}} \e{(n_1+\cdots+n_k)\frac{a}{q}}\frac{\tilde d(n_1,\dots,n_k)}{n_1\cdots n_k} +O\pr{\frac{(A\log X)^{2k}q^{1+2\eps_1}}{X}},
}
where $B^*_{X}:=[-X,X]\cap\Z_{\neq0}$, 
\est{
\tilde d(n_1,\dots,n_k):=d(|n_1|)\cdots d(|n_k|)g\pr{ {|n_1|}/X}\cdots g\pr{ {|n_k|}/X}
}
and  $A$ denotes an absolute positive constant that might change from line to line. Thus, by M\"obius inversion formula and the orthogonality of additive characters we have
\est{
\frac1{\varphi(q)}\sum_{\substack{a=1,\\(a,q)=1}}^{q}D_{\sin}(1, a/q)^k&=\sum_{\ell|q}\frac{\mu(q/\ell)}{\varphi(q)}\sum_{a=1}^{\ell}D_{\sin}(1, a/\ell)^k\\[-0.8em]
&=(2i)^{-k}\sum_{\ell|q}\frac{\mu(q/\ell)\ell}{\varphi(q)}\hspace{-1.1em}\sum_{\substack{n_1,\dots,n_k\in B^*_{2X},\\n_1+\cdots+n_k\equiv 0\mod \ell}} \hspace{-1.6em}\frac{\tilde d(n_1,\dots,n_k)}{n_1\cdots n_k} +O\pr{\frac{(A\log X)^{2k}q^{1+3\eps_1}}{X}}.
}
The contribution of the terms with $n_1+\cdots+n_k\neq0$ is bounded by
\est{
\sum_{\ell|q}\frac{\ell}{\varphi(q)}\sum_{\substack{n_1,\dots,n_k\in B^*_{2X},\\0\neq n_1+\cdots+n_k\equiv 0\mod \ell}}\hspace{-1em} \frac{|\tilde d(n_1,\dots,n_k)|}{|n_1\cdots n_k|} 
&\ll_\eps \sum_{\ell|q}\frac{k\ell}{\varphi(q)}\hspace{-0.7em}\sum_{\substack{n_1,\dots,n_k\in B^*_{2X},\\ |n_1|\geq \ell/k,\\ 0\neq n_1+\cdots+n_k\equiv 0\mod q}}\hspace{-1.5em} \frac{A^kX^{\eps}d(|n_2|)\cdots d(|n_k|)  }{|n_1\cdots n_k|}\\
&\ll_\eps q^{-1+\eps}A^kX^{\eps}(\log X)^{2k},
}
since 
\est{
\sum_{\substack{\ell/k\leq n \leq 2X,\\ n\equiv c\mod \ell}}\frac1n\ll \frac{k}{\ell}+\frac{\log X}{\ell}.
}
The contribution of the diagonal term $n_1+\cdots+n_k=0$ is 
\est{
\sum_{\ell|q}\frac{\mu(q/\ell)\ell}{\varphi(q)} (2i)^{-k}\hspace{-1em}\sum_{\substack{n_1,\dots,n_k\in B^*_{2X},\\n_1+\cdots+n_k= 0}} \frac{\tilde d(n_1,\dots,n_k)}{|n_1\cdots n_k|}&=\frac{\pi^{2k}}{2^k} H_k\sum_{\ell|q}\frac{\mu(q/\ell)\ell}{\varphi(q)}+O_\eps((Ak)^{2k}q^\eps X^{-\frac12+\eps})\\[-0.5em]
&=\frac{\pi^{2k}}{2^k}H_k+O_\eps((Ak)^{2k}q^\eps X^{-\frac12+\eps}),
}
since
\est{
\sum_{\substack{ (n_1,\dots,n_k)\in (\Z_{\neq0})^k\setminus (B^*_{2X})^k ,\\n_1+\cdots+n_k= 0}} \frac{\tilde d(n_1,\dots,n_k)}{|n_1\cdots n_k|}\ll kA^kX^{-\frac12+\eps}\sum_{\substack{  n_2,\dots,n_k\in\Z_{\neq0}}} \frac{ d(|n_2|)\cdots d(|n_k|)}{|n_2\cdots n_{k}|^{1+\frac{1/2}{k-1}}}\ll (Ak)^{2k} X^{-\frac12+\eps},
}
where we used $|n_1|\geq X^\frac12 |n_2\cdots n_k|^{\frac1{2(k-1)}}$ if $ |n_2|,\dots,|n_k|,X\leq |n_1|$. Theorem~\ref{cstare} then follows upon choosing $X=q^3$.

\section{Proof of Proposition~\ref{mp} and Theorem~\ref{cbg}}
The following lemmas give Proposition~\ref{mp} in the cases of $x\in\Q$ and $x\notin\Q$ respectively.
\begin{lemma}\label{caf}
Let $(a,q)=1$, $q>0$ and let $v_0,\dots,v_r$ be the partial denominators of the continued fraction expansion $a/q=\pa{a_0;a_1,\dots,a_r}$. 
Then
\est{
D_{\sin}(1,a/q)=-\frac{\pi^2}2\sum_{\ell=1}^{r}\frac{(-1)^\ell}{v_\ell}\pr{\frac1{\pi v_{\ell}}+\psi\pr{\frac{v_{\ell-1}}{v_\ell}}}.
}
\end{lemma}
\begin{proof}
Write $\frac bq:=(-1)^{r+1}\frac{\overline a}{q}$. Then, one has the continued fraction expansion $b/q=\pa{0;b_1,\dots,a_r}= \pa{0;a_r,\dots,a_1}$. Moreover, the Euclid algorithm for $b/q$ gives
\est{%\label{aaa}
&y_1=q,\qquad y_2=b,\\
&y_{n-1}=b_{n-1}y_{n}+y_{n+1},\qquad n=1,\dots r+1,\\
}
with  $y_{r+1-\ell}=v_{\ell}$, where $v_\ell$ is the $\ell$-th partial quotient of $a/q$ (as usual we put $v_{-1}:=0$). Thus, applying repeatedly the reciprocity formula~\eqref{rf} and using that $c_0(1)=0$, we obtain
\est{
\frac1qc_0\pr{b/q}
&=-\sum_{m=1}^{r}\frac{(-1)^m}{\pi y_{m}^2}-\sum_{m=1}^{r}\frac{(-1)^m}{y_m}\psi\pr{\frac{y_{m+1}}{y_m}}=\sum_{\ell=1}^{r}\frac{(-1)^{\ell+r}}{v_\ell}\pr{\frac1{\pi v_\ell}+\psi\pr{\frac{v_{\ell-1}}{v_\ell}}}
}
and the Lemma follows by~\eqref{sf} since $D(s,x)=-D(s,-x)$.
\end{proof}

\begin{lemma}\label{fflrm}
Let $x\in\R\setminus\Q$ and assume $x$ has continued fraction expansion $x=\pa{a_0;a_1,a_2,\dots}$ with partial quotients $v_0,v_1,v_2,\dots$ Then
\es{\label{if}
D_{\sin}(1,a/q)=-\frac{\pi^2}2\sum_{\ell=1}^{\infty}\frac{(-1)^\ell}{v_\ell}\pr{\frac1{\pi v_{\ell}}+\psi\pr{\frac{v_{\ell-1}}{v_\ell}}},
}
whenever $D_{\sin}(1,a/q)$ is defined. Moreover, writing 
\est{
D_{X}(1,x):=\sum_{n\leq X}\frac{d(n)\sin(2\pi nx)}{n}, \qquad  S(x):=\sum_{n=1}^\infty \frac{\log v_{n+1}}{v_n},
}
we have $D_{X}(1,x)\ll S(x)$ and $D(1,x)\ll S(x)$, uniformly in $x\in[0,1]\setminus\Q$, $X\geq2$.
\end{lemma}
\begin{proof}
For a large positive constant $B\geq5$, let $\xi_{r}=v_{r}(\log v_{r})^B$ and let $R$ be the minimum integer such that $\xi_{R}\leq X$. We can split $D_X(1,x)$ into
\es{\label{spt}
D_{X}(1,x)=\sum_{n\leq  \xi_R}d(n)\frac{\sin (2\pi nx)}{n}+\sum_{\xi_{R}\leq n\leq  X}d(n)\frac{\sin (2\pi nx)}{n}.
}
The second addend can be treated using the work of de la Bret\`eche and Tenenbaum~\cite{dBT}. Indeed, by partial summation, if $B$ is sufficiently large we have 
\est{
\sum_{\xi_{R}\leq n\leq  X}d(n)\frac{\sin (2\pi nx)}{n}&=\sum_{\xi_{R}\leq n\leq  X}d(n)\frac{\sin (2\pi nx)}{X}+\int_{\xi_{R}}^X\sum_{\xi_{R}\leq n\leq  t}d(n)\frac{\sin (2\pi nx)}{t^2}\, dt\\
&=O\pr{\frac{\log (v_{R+1})}{v_{R}}+\frac1{\log (v_{R})}}
}
by~(11.1) and~(11.4) of~\cite{dBT}. For the first addend of the right hand side of~\eqref{spt}, we first observe that 
\est{
\sum_{n\leq \xi_R}\frac{d(n)\sin(2\pi nx)}{n}=\sum_{n\leq \xi_R}\frac{d(n)\sin(2\pi n\frac{u_R}{v_R})}{n}+O\pr{\frac{(\log v_R)^{1+B}}{v_{R+1}}}
}
since $|x-u_R/v_R|\leq (v_Rv_{R+1})^{-1}$. Moreover, we observe that  by Mellin's formula we have
\est{
\sum_{n\leq \xi_R}d(n)\frac{\sin(2\pi n\frac{u_R}{v_R})}{n}&=D_{\sin}(1,{u_R}/{v_R})+\frac1{2\pi i}\int_{\substack{C}}D_{\sin}(1+s,{u_R}/{v_R})\xi_R^s\frac{ds}{s}+O((\log \xi_{R})^2/T)\\
&=D_{\sin}(1,u_R/v_R)+O(v_RT^{\frac12}(\log v_R)^2/\xi_R+(\log v_R)^2/T)\\
&=D_{\sin}(1,u_R/v_R)+O((\log v_R)^{-1}),
}
where $C$ denotes the line from $s=(-1-\frac1{\log x})-iT$ to $s=(-1-\frac1{\log x})+iT$ with $T=(\log v_{R})^4$ and where to bound the integral we used the functional equation~\eqref{fe} and a trivial bound.

Finally, by Lemma~\ref{caf} we have
\est{
D_{\sin}(1,u_R/v_R)=-\frac{\pi^2}2\sum_{\ell=1}^{R}\frac{(-1)^\ell}{v_\ell}\pr{\frac1{\pi v_{\ell}}+\psi\pr{\frac{v_{\ell-1}}{v_\ell}}}
}
and thus
\est{
D_{X}(1,x)=-\frac{\pi^2}2\sum_{\ell=1}^{R}\frac{(-1)^\ell}{v_\ell}\pr{\frac1{\pi v_{\ell}}+\psi\pr{\frac{v_{\ell-1}}{v_\ell}}}+O\pr{\frac{\log (v_{R+1})}{v_{R}}+\frac1{\log (v_{R})}}.
}
As $X\rightarrow\infty$, we have $v_{R}\rightarrow\infty$ and, by Theorem 4.4 of~\cite{dBT}, $\log(v_{R+1})/v_R\rightarrow 0$ if (and only if) the series defining $S(1,x)$ converges and thus we obtain~\eqref{if}. The second assertion of the Lemma then follows by~\eqref{aps}.
\end{proof}

We need two results from Khinchin's book on continued fractions~\cite{Khi}.
\begin{lemma}\label{fk}
For all $x\in\R\setminus\Q$ and all $n\geq1$ we have $v_n\geq 2^{\frac{n-3}2}$.
\end{lemma} 
\begin{proof}
This is Theorem~12 of~\cite{Khi}.
\end{proof}
The following lemma is a minor refinement of Theorem~31 of~\cite{Khi}.

\begin{lemma}\label{afra}
Let $K\geq 1$. Then for all $\eps>0$, there exists $B_\eps>0$ such that
\est{
E(K):=\tn{meas}\pr{\pg{x\in[0,1]\mid v_r(x)\geq Ke^{B_\eps r}\ \exists r\geq1}}\ll_\eps K^{-1+\eps}.
}
\end{lemma}
\begin{proof}
Proceeding as in the proof of Theorem~ 31 of~\cite{Khi}, we see that for all $n,B\geq1$  we have 
\est{
\tn{meas}\pr{E_n(K)}\ll \frac{2^n}{Ke^{Bn}}\sum_{\ell=0}^{n-1}\frac{(\log( Ke^{Bn}))^\ell}{\ell!},
}
where $E_n(K):=\pr{\pg{x\in[0,1]\mid v_n(x)\geq Ke^{Bn}}}$
(this is the first equation on page 68 of~\cite{Khi}, with $g=Ke^{n}$). Now, if $K\leq e^{Bn}$ and $B$ is large enough, then 
\est{
\tn{meas}\pr{E_n(K)}&\ll \frac{2^n}{Ke^{Bn}}\sum_{\ell=0}^{n-1}\frac{(\log( Ke^{Bn}))^\ell}{\ell!}\ll \frac{2^n}{Ke^{Bn}}\sum_{\ell=0}^{n-1}\frac{(2Bn)^\ell}{\ell!}\ll\frac{n}{K}\frac{(4ne^{-B})^n}{n!}\ll \frac{e^{-Bn/2}}{K},
}
where we used $C^\ell /\ell!\ll C^n /n!$, valid for $0\leq \ell\leq n\leq C$, and Stirling's formula. 
In the same way, if $K>e^{Bn}$ and $B$ is large enough, then
\est{
\tn{meas}\pr{E_n(K)}&\ll \frac{n(4\log K)^n}{Kn!e^{Bn}}\ll \frac{(e^{-B/2}\log K)^n}{Kn!}.
}
Thus, we have
\est{
E(K)\leq\sum_{n=1}^\infty\tn{meas}\pr{E_n(K)}\ll_\eps K^{-1+e^{-B/2}}
}
and the Lemma follows.
\end{proof}
\begin{corol}%\label{cavd}
For $K\geq 1$, we have
\as{\label{mwq}
\tn{meas}\pr{\pg{x\in[0,1]\mid |S(x)|> K}}= O( e^{-\delta K})
}
for some $\delta>0$. 
\end{corol}
\begin{proof}
By Lemma~\ref{fk} and~\ref{afra}, if $x\in[0,1]\setminus E(e^K)$ we have
\est{
S(x)\ll\sum_{n=1}^\infty \frac{B_\eps n+ K}{2^{n/2}}\ll_\eps  1+ K
}
and~\eqref{mwq} follows.
\end{proof}

We can now prove Theorem~\ref{cbg} and Corollary~\ref{mcrm}. 
\begin{proof}[Proof of Theorem~\ref{cbg} and Corollary~\ref{mcrm}]
Expressing the linear constraint in the definition of $H_k$ as an integral, we see that
\est{
H_k&=(i\pi^{2})^{-k}\lim_{X\rightarrow \infty}\int_{0}^1\sum_{\substack{-X\leq n_1,\dots,n_k\leq X,\\ n_1\cdots n_k\neq0}} \frac{\e{(n_1+\cdots n_k)x}d(|n_1|)\cdots d(|n_k|)}{n_1\cdots n_k}\,dx\\
&=\lim_{X\rightarrow \infty}\frac{2^k}{\pi^{2k}}\int_{0}^1D_X(1,x)^k\,dx=\frac{2^k}{\pi^{2k}}\int_{0}^1D(1,x)^k\,dx,
}
where the exchange of order of summation and integration is justified by the dominated convergence theorem, since $D_X(1,x)\ll S(x)$ by Lemma~\ref{fflrm} and, by~\eqref{mwq},
\est{
\int_{0}^1S(x)^kdx\leq \sum_{L=1}^{\infty}{\int_0^1} \chi_L(x)L^kdx\ll \sum_{L=1}^{\infty}L^ke^{-(L-1)\delta}\ll A^kk!,\
}
for some $\delta,A>0$ and where $\chi_L$ is the characteristic function of the set $\{x\mid L-1\leq S(x)\leq L\}$.

Since we also have $D(1,x)\ll S(x)$, the above computation also proves Theorem~\ref{cbg}. Corollary~\ref{mcrm} then follows since $\sum_{k=1}H_kt^k/k!$ has a positive radius of convergence.
\end{proof}

Finally, we prove Corollary~\ref{finc}.  
\begin{proof}[Proof of Corollary~\ref{finc}]
By Lemma~\ref{afra}, we can find some absolute constants $K,B$ such that the set
\est{
S(x_1,x_2,\kappa):=\{x=\pa{0;1,\dots,1,x_1,x_2,a_{\kappa+3},a_{\kappa+4},\dots}\mid 1\leq a_{\ell}\leq Ke^{B\ell},\forall  \ell\geq \kappa+3\}
}
has positive measure for any $x_1,x_2,\kappa \in\Z_{>0}$. Thus, to prove the corollary it is enough to show that for any $z\in\R$, $\eps>0$ there exist integers
 $x_1,x_2,\kappa\geq1$ such that $z< D(1,x)\leq z+\eps$ for all $x\in S(x_1,x_2,\kappa)$.

Now, if $x\in S$ then by~\eqref{csa} we have
$D_{\sin}(s,x)=\mathcal A+\mathcal B+\mathcal C,
$
where
\est{
\mathcal A&=-\frac{\pi^2}2\sum_{\ell=1}^{\kappa}\frac{(-1)^\ell}{v_\ell}\pr{\frac1{\pi v_{\ell}}+\psi\pr{\frac{v_{\ell-1}}{v_\ell}}},\\
\mathcal B&=-\frac{\pi^2}2\sum_{\ell=\kappa+1}^{\kappa+2}\frac{(-1)^\ell}{v_\ell}\pr{\frac1{\pi v_{\ell}}+\psi\pr{\frac{v_{\ell-1}}{v_\ell}}},\\
}
and, by~\eqref{aps} and Lemma~\ref{fk},
\est{
\mathcal C&\ll\sum_{\ell=\kappa+3}^{\infty}\frac{\log v_{\ell}}{v_{\ell-1}}\ll \sum_{\ell=\kappa+3}^{\infty}\frac{\log K+B\ell}{2^\ell}\leq \eps/10,
}
provided that $\kappa=\kappa_\eps$ is large enough. Now, from the relation $v_n=a_nv_{n-1}+v_{n-2}$, we see that 
$v_{\kappa+1}=x_{1}v_{\kappa}+v_{\kappa-1}$ and $v_{\kappa+2}= x_{2}(x_{1}v_{\kappa}+v_{\kappa-1})+v_{\kappa}$. Thus, by~\eqref{aps} we have 
\est{
\mathcal B&=\alpha_\kappa\log (x_1)-\beta_\kappa\frac{\log (x_2)}{x_1+\gamma_\kappa}+o(1)\\
&=\alpha_\kappa\log (x_1)-\beta_\kappa\frac{\log (x_2)}{x_1}+o(1)+O\pr{\frac{\log (x_{2})}{x_1^2}},
}
as $x_1,x_2\rightarrow\infty$ (and $\kappa$ fixed), for some $\alpha_\kappa,\beta_\kappa,\gamma_\kappa\neq0$. Now, if we pick
\est{
x_2:=[\exp\pr{\beta_\kappa^{-1}x_1(\alpha_\kappa \log (x_1)+\mathcal A-z-\eps/2)}],
}
then $\mathcal B=z-\mathcal A+\eps/2+o(1)$. Thus, if $x_1$ is large enough we have 
\est{
z< D(1,x)\leq z+\eps
}
and the corollary follows.
\end{proof}
\begin{remark}
We remark that a modification of this proof in the spirit of~\cite{Hic} would have given Corollary~\ref{den} directly.
\end{remark}

\appendix

\end{document}